\documentclass[10pt]{amsart} 
\usepackage{graphicx,latexsym,bm,amsmath,amssymb,verbatim,multicol,lscape}
\newtheorem{thm}{Theorem} [section]
\newtheorem{cor}[thm]{Corollary}
\newtheorem{lem}[thm]{Lemma}

\theoremstyle{definition}
\newtheorem{defn}[thm]{Definition}
\theoremstyle{remark}

\numberwithin{equation}{section}

\begin{document}
\title{The universal Kummer congruences}%
\author{Shaofang Hong}
\address{Yangtze Center of Mathematics,  Sichuan University, Chengdu 610064, P.R. China
and Mathematical College, Sichuan University, Chengdu 610064, P.R. China}
\email{sfhong@scu.edu.cn; s-f.hong@tom.com; hongsf02@yahoo.com}
\author{Jianrong Zhao}
\address{School of Economic Mathematics, Southwestern University of
Finance and Economics, Chengdu 610074, P.R. China}%
\email{mathzjr@foxmail.com}
\author{Wei Zhao}
\address{Science and Technology on Communication Security Laboratory,
Chengdu 610041, P.R. China}%
\email{zhaowei9801@163.com}
\thanks{The research of Hong was supported partially by the National Science
Foundation of China Grant \# 10971145 and by the Ph.D. Programs Foundation
of Ministry of Education of China Grant \#20100181110073
}
\date{\today}%
\dedicatory{Dedicated to Professor Qi Sun on the Occasion of his 75th Birthday}%
\begin{abstract}
Let $p$ be a prime. In this paper, we present a detailed $p$-adic
analysis to factorials and double factorials and their congruences.
We give good bounds for the $p$-adic sizes of the coefficients of
the divided universal Bernoulli number ${{\hat B_n}\over n}$ when
$n$ is divisible by $p-1$. Using these we then establish the
universal Kummer congruences modulo powers of a prime $p$ for the
divided universal Bernoulli numbers ${{\hat B_n}\over n}$ when $n$
is divisible by $p-1$. 
\end{abstract}
\maketitle

\section{Introduction}
Bernoulli numbers occur in many parts of number theory. Let $n\ge 0$
be an integer. Then the $n$-th {\it Bernoulli number} is defined by the
following formula:
\begin{align}\label{01}
\frac{t}{e^t-1}=\sum^{\infty }_{n=0}\frac{B_n}{n!}t^n.
\end{align}
From (\ref{01}) one can read that $B_1=-1/2$ and $B_n=0$ for all odd
$n>1$. The first few values for even $n$ are: $B_0=1, B_2=1/6,
B_4=-1/30, B_6=1/42,$ etc. The periodic behavior of the divided
Bernoulli numbers ${{B_n}\over n}$ is closely related to the
existence of a $p$-adic zeta function \cite{[Mu]}. The classical
Kummer congruences \cite{[Ku]} concern the congruence relations
among the divided Bernoulli numbers ${{B_n}\over n}$. In fact, they
state that if $p$ is a prime and $(p-1)\nmid n$ and $n\equiv m
\pmod {p-1}$, then $\frac{B_{n}}{n}\equiv\frac{B_{m}}{m}\pmod p$.
One can prove this congruence by means of $p$-adic measures and
$p$-adic integration \cite{[Y1], [Y2]}. In \cite{[BCRS]}, Baker et
al. established some global-local Kummer congruences. There are also
many other elegant and useful congruences, such as Wilson's theorem
\cite{[A4], [L]}, Fermat's little theorem \cite{[Ko], [L]},
Wolstenholme's theorem \cite{[Z]}, Lucas' congruence
\cite{[Ro]} and Glaisher's congruence \cite{[G1], [G2], [Ho]}.
In this paper, we will mainly be concerned with
the universal Kummer congruences.

In 1989, Clarke \cite{[C]} introduced the concept of universal
Bernoulli numbers. Assume that $c_{1}, c_{2}, \ldots $ are
indeterminates over ${\bf Q}$. Then let $F(t)=t+c_{1}{{t^{2}}\over 2}
+c_{2}{{t^{3}}\over 3}+\cdots$ and let $G(t)=F^{-1}(t)$ be the compositional
formal power series inverse of $F(t)$, namely $F(G(t))=G(F(t))=t$.
The {\it universal Bernoulli numbers} $\hat B_{n}$ are defined by
\begin{align}\label{02}
{t\over {G(t)}}=\sum_{n=0}^{\infty}\hat B_{n}{{t^{n}}\over {n!}}.
\end{align}
Obviously we have $\hat B_{n}\in {\bf Q}[c_{1},c_{2},\ldots ,c_{n}]$.
Actually $\hat B_{n}$ is a non-trivial {\bf Q}-linear combination of
all the monomials of weight $n$, where $c_{i}$ has weight $i$. So
$\hat{B_{n}}$ is the sum of $p(n)$ monomials, where $p(n)$ is the
partition function. Recently, Tempesta \cite{[T]} introduced the
universal higher-order Bernoulli polynomials and universal Bernoulli
$\chi$-numbers.

Substituting $c_{i}=(-1)^{i}$, we get $F(t)=\log(1+t)$ so that
$G(t)=e^{t}-1$ and we obtain the classical Bernoulli numbers
$B_{n}=\hat B_{n}$. Miller \cite{[Mi]} investigated the
specialization where $c_i$ is the equivalence class of the complex
projective space and proved that for this specialization, if $k$ is
odd and $k\ne 1$, then $\hat B_k/k\in L$, where $L$ is the {\it
Lazard ring}, a subring of ${\bf Q}[c_1, c_2, \ldots ]$. Clarke
\cite{[C]} showed that the divided universal Bernoulli number
${{\hat B_n}\over n}$ is $p$-integral if $(p-1)\nmid n$ which forms
part of his universal von Staudt theorem. Adelberg \cite{[A3]} set
up the universal Kummer congruences modulo a prime $p$ for the divided
universal Bernoulli numbers ${{\hat B_n}\over n}$ when $(p-1)\nmid n$.
Consequently Adelberg  \cite{[A4]} obtained the universal
Kummer congruences modulo powers of a prime $p$ for the case
$(p-1)\nmid n$. Adelberg, Hong and Ren \cite{[AHR]} established
the universal Kummer congruences modulo a prime $p$ for the divided
universal Bernoulli numbers when $(p-1)|n$. Besides, Adelberg, Hong
and Ren \cite{[AHR]} gave a simple proof to Clarke's 1989 universal
von Staudt theorem \cite{[C]} which generalized the theorems of
Dibag \cite{[D]}, Ray \cite{[R]}, Katz \cite{[Ka]} and Hurwitz
\cite{[Hu]}. It is natural to ask the question of establishing the
universal Kummer congruences modulo powers of a prime $p$ for the
divided universal Bernoulli numbers when $(p-1)|n$.

In the present paper, our main goal is to investigate the above problem.
We will exploit the universal Kummer congruence modulo powers of a
prime $p$ for the divided universal Bernoulli numbers for the
remaining case $(p-1)|n$. It is divided into two subcases: the even prime
$2$ and odd primes $p$ such that $(p-1)|n$. For this purpose, we
need to estimate $p$-adic valuations of the coefficients of the
divided universal Bernoulli number ${{\hat B_n}\over n}$ when $n$ is
divisible by $p-1$. This paper is organized as follows. In Section 2,
we present notations and some preliminary results. As a consequence, in
Section 3, we establish the universal Kummer congruences modulo
powers of odd primes $p$ when $(p-1)|n$. Consequently, in Section 4, we
treat the even prime $p=2$ case. We provide a detailed 2-adic analysis to
many kinds of factorials and double factorials. Finally, in Section 5,
we set up the universal Kummer congruences modulo powers of
2. Our result extends and strengthen the Adelberg-Hong-Ren theorem.
It also extends Clarke's theorem and complements Adelberg's modulo
prime powers result.

\section{Preliminaries}

If $u=(u_1,u_2,\ldots )\in \mathbf{N}^{\infty}$ with $u_i=0$ if $i$ is 
sufficiently large and $\omega(u):=\sum iu_i$, we identify $u$ with a 
partition of $\omega(u)$, where $u_i$ is the number of occurrences of the 
part $i$ in the partition. If $d(u):=\sum u_i$, then $d(u)$ is the number
of parts in the partition. We call $\omega(u)$ the {\it weight} of
$u$ and $d(u)$ the {\it degree} of $u$. If $u_i=0$ for $i>n$, we
write $u\in \mathbf{N}^n$. Let $k\ge 1$ and $N\ge 1$ be integers and
$l=k2^N$.

As usual, we let $v=v_p$ be the {\it normalized $p$-adic valuation}
of $\mathbf{Q}$, i.e., $v(a)=b$ if $p^b\mid a$ and $p^{b+1}\nmid a$. We can extend $v$ to
$\mathbf{Q}[c_1,c_2,\ldots ]$ by $v(\sum a_uc^u)=\min\{v(a_u)\}$ when
$u=(u_1,\ldots ,u_n)\in \mathbf{N}^n$ and $c_u=c_1^{u_1}\ldots c_n^{u_n}$.
By the Lagrange inversion \cite{[C]}, we
have

\begin{align}\label{1}
\frac{\widehat{B}_n}{n}=\sum_{\omega=n}\tau_uc^u,
\end{align}
where $u=(u_1,\ldots ,u_n)\in \mathbf{N}^n$, $\omega=\omega(u)$,
$d=d(u)$, $c^u=c_1^{u_1}\ldots c_n^{u_n}$, $\gamma
_u=2^{u_1}\ldots (n+1)^{u_n}u_1!\ldots u_n!$ and
\begin{align}\label{2}
\tau_u=\frac{(-1)^{d-1}(n+d-2)!}{\gamma_u}.
\end{align}

For any positive odd integer $a$, we define the {\it double
factorial} $a!!$ of $a$ by $a!!=\prod_{1\le k\le a, (2,k)=1}k.$ That
is, $a!!=a\cdot(a-2)\cdot \ldots \cdot 3\cdot1$. For a real number $x$,
define $\lceil x\rceil $ to be the least integer greater than $x$
and $\lfloor x\rfloor $ to be the greatest integer less than $x$.
Then $\lceil x\rceil +\lceil y\rceil \ge \lceil x+y\rceil $, $\lceil
n+x\rceil =n+\lceil x\rceil $, $\lfloor x\rfloor +\lfloor y\rfloor
\le \lfloor x+y\rfloor $ and $\lfloor n+x\rfloor =n+\lfloor x\rfloor
$ for any real numbers $x$ and $y$ and any integer $n$.

We will freely use the standard results listed in the following
Lemma \ref{A3 lemma 2.0}.
\begin{lem}
\label{A3 lemma 2.0} We have
\begin{align}
v((ab)!)&\geq v(a!)+v(b!). \label{eqn 0}\\
  v((lp)!)&=l+v(l!). \ {\it Moreover,} \  v((lp^t)!)=\frac{l(p^t-1)}{p-1}+v(l!).\label{eqn 1}\\
  v(a!)=v((\lfloor a/p \rfloor p)!)&=(a-s(a))/(p-1),\label{eqn 3}
\end{align}
where $s(a):=\sum_{i=0}^ta_i$ is the base $p$ digit sum if
$a=\sum_{i=0}^ta_ip^i$ with digits $0\le a_i\le p-1$, and
\begin{align}\label{eqn 2}
v(a!)\leq (a-1)/(p-1)  {\it~ if ~} a>0.
\end{align}
\end{lem}

\begin{lem}
\cite{[A4]} \label{A3 lemma 2.1} If $p$ is an odd prime and
$N=v(l)$, then
  $$(lp)!/(l!p^l)\equiv (-1)^l \pmod {p^{N+1}}.$$
\end{lem}

\begin{lem}
\cite{[A4]} \label{A3 lemma 2.2} $v((\sum h_jp^j)!)\geq
\sum(jh_j+v(h_j!)).$
\end{lem}

\begin{lem}\cite{[A4]}\label{A3 lemma 2.3} If $p$ is an
odd prime and $0<k\leq p$ then
$$v(a!)\geq v(a+k)$$ unless $a=p-k$, in which case
$v(a!)=v(a+k)-1$.
\end{lem}

Clarke \cite{[C]} proved some congruences about factorials. Clarke
and Jones \cite{[CJ]} obtained some stronger congruences about
factorials while Adelberg, Hong and Ren \cite{[AHR]} strengthened
Clarke's congruence for the $p=2$ case. But for our purpose, we need
the following result due to Adelberg \cite{[A4]}.

\begin{lem} \cite{[A4]} \label{A3 lemma 2.4}Let
$v(l)=N.$ Then

{\rm (i)} $((l+q)p)!/((l+q)!p^{l+q})\equiv
(-1)^l(qp)!/(q!p^q)\pmod{p^{N+1}}$.

{\rm (ii)} $((l+q)p+a)!/((l+q)!p^{l+q})\equiv
(-1)^l(qp+a)!/(q!p^q)\pmod{p^{N+1}}$.

{\rm (iii)} If $a\geq ep$, then the congruence {\rm (ii)} holds $\pmod
{p^{N+e}}$.

{\rm (iv)} If $a\geq (e+1)p$, then
$$
((l+q)p+a)!/((l+q)!p^{l+q+e})\equiv
(-1)^l(qp+a)!/(q!p^{q+e})\pmod{p^{N+1}}.
$$
\end{lem}

\section{Universal Kummer congruences modulo odd prime powers}

In the present section, we treat the odd prime case. We set up the
universal Kummer congruences modulo odd prime powers. We begin with
the following concept. Note that it is different from Definition 4.1
in \cite{[A4]}.

\begin{defn} A partition $u$ is called {\it reduced} if there is at
most one part $g\in \mathbf{N}$ such that $g\neq p^\alpha-1$, $u_g=1$
and if $i\neq g$ and $u_i\neq 0$, then $i=p^\alpha-1$.
\end{defn}

 \begin{lem}\label{lem -1}
Assume  $n=(m+i)p-i=i(p-1)+mp$ where $i\geq 0$. Let $w(u)=n$ and
suppose that $d(u)\leq i+1$. Then there exists a reduced partition
$u'$ such that $w(u')=n$, $d(u')\leq i+1$ and $v(\tau_u)\geq
v(\tau_{u'})$.
 \end{lem}

{\it Proof.} We first define a partition $u''$: If $t\ne
p^\alpha-1$, then $u_t'':=0$; if $t=p^\alpha-1$, then
\begin{align}\label{0000}
u''_t=u''_{p^\alpha-1}:=\sum_{\varepsilon\in \mathbf{Z}^+,~p\nmid
\varepsilon}u_{\varepsilon p^\alpha-1}+\delta(\alpha)\cdot
\sum_{u_j\ge p, ~p\nmid (j+1)}
\Big(\Big\lceil\frac{u_j}{p-1}\Big\rceil-1\Big),
 \end{align}
where $\delta(\alpha)=1$ for  $\alpha=1$ and $\delta(\alpha)=0$ for
$\alpha\ge 2$.  In fact, we construct $u''$ as follows:

(i) If $u_t\neq 0$ with $t=\varepsilon p^\alpha -1$ and $p
 \nmid  \varepsilon$ and $\varepsilon >1$, let $u''_t=0$ and
$u''_k=u_k+u_t$ where $k=p^\alpha -1$, i.e., transfer $u_t$ to the
part $p^\alpha -1$.

(ii) If $u_t\geq p$, where $p\nmid (t+1)$, let $u''_t=0$ and
$u''_{p-1}=u_{p-1}+\lceil u_t/(p-1)\rceil
  -1$, i.e., transfer to the part $p-1$.

(iii) If $0<u_t<p$ and $p\nmid (t+1)$, let $u''_t=0$, i.e., delete
the part $t$ from the partition. The partition $u''_t$ is formed by
considering all parts, and should be thought of as loading the parts
where $i=p^\alpha -1$. All other parts of $u$ where (i)-(iii) do not
apply are unchanged. The partition $u''_t$ can be constructed from
$u$ one part at a time or all at once.

Clearly $w(u'')\leq n$ and if $u_i\neq 0$ then $i=p^\alpha -1$.
Observe that $v(\gamma_u'')\geq v(\gamma_u)$ by (\ref{eqn 1}) and
(\ref{eqn 2}), and that $d(u'')=d(u)$ if all modifications are of
type (i), where otherwise $d(u'')<d(u)$.

Next let $g=n-w(u'')\geq 0$. If $g=0$, let $u'=u''$. Then $u'$ is
just what we need.

If $g>0$, let $u'_g=u''_g+1$ and $u'_j=u''_j$ if $j\neq g$. Then
$u'$ is reduced, $w(u')=n$, $d'=d(u')=d(u'')+1$ and
$v(\gamma_{u'})\geq v(\gamma_{u''})\geq v(\gamma_u)$. If
$d(u'')<d(u)$, then $d(u')\leq d(u)\leq i+1$. Hence $v(\tau_u)\geq
v(\tau_{u'})$. Finally assume $d(u'')=d(u)$. In this case all
modifications are of type (i), so $n=\sum iu_i=\sum (\varepsilon
p^\alpha-1)u_{\varepsilon p^\alpha-1}$ and $d=\sum u_{\varepsilon
p^\alpha-1}$. But $n=(m+i)p-i$, thus $i\equiv\sum u_{\varepsilon
p^\alpha-1}=d(u)\pmod p$. Since $1\leq d(u) \leq i+1$, it follows
that $d(u)\leq i$ and $d(u')=d(u)+1\leq i+1$. Also
$n+d-2=\sum\varepsilon p^\alpha u_{\varepsilon p^\alpha-1}-2$. Hence
$v((n+d-2)!)=v((n+d'-2)!)$  by (\ref{eqn 3}), which means
$v(\tau_u)\geq v(\tau_{u'})$. The proof is complete.\hfill$\Box$

\begin{lem}\label{lem -2} Let $n=(m+i)p-i$ with $m=s(p-1)$. Let u be reduced,
with $w=n$ and $d\leq i+1$. Then $v(\tau_u)\geq s(p-2)-1$.
\end{lem}
\noindent{\it Proof.} Let $h_j=u_{p^j-1}$ with $j\geq 1$. Consider
the following cases:

{\sc Case 1.} If $u_t\neq 0$, then $t=p^\alpha -1$. In this case,
$d=\sum_{j\geq 1} h_j$ and $n=\sum_{j\geq 1} h_j(p^j-1)=(m+i)p-i$.
Then $i\equiv \sum_{j\geq 1} h_j=d$ (mod $p$). Since $d\leq i+1$, we
may let $i=d+kp$ with $k\geq 0$. Then
$$n+d-2=\sum_{j\geq 1} h_jp^j-2=(m+i)p-i+d-2=(m+i-k)p-2.$$
It implies that  $k=m+i-\sum_{j\geq 1} h_jp^{j-1}$.  It is easy to
see that $v(\gamma_u)=\sum_{j\geq 1}(jh_j+v(h_j!))$. Hence by Lemma
\ref{A3 lemma 2.2}
\begin{align}\label{eqn 4}
    v(\tau_u)&=v(((m+i-k)p-2)!)-\sum_{j\geq 1} (jh_j+v(h_j!))\\
    &=m+i-k-1+v((m+i-k-1)!)-\sum_{j\geq 1} (jh_j+v(h_j!))\nonumber\\
    &=m+i-k-1+v((\sum_{j\geq 1} h_jp^{j-1}-1)!)-\sum_{j\geq 1} ((j-1)h_j+v(h_j!))-d\nonumber\\
    &=m-1+k(p-1)+\beta\nonumber,
\end{align}
where $$\beta=v((\sum_{j\geq 1} h_jp^{j-1}-1)!)-\sum_{j\geq 1}
((j-1)h_j+v(h_j!)).$$

For $j\geq 3$ and $h_j\neq 0$, we have
\begin{align*}
v((h_jp^{j-1}-1)!)&=h_jp^{j-2}-1+v((h_jp^{j-2}-1)!)\\
&\ge h_jp^{j-2}-1+v((h_jp-1)!)\\ &\geq(j-1)h_j+v(h_j!)
\end{align*}
since $p^{j-2}-j+1\ge1$. So if there exists a $j\geq 3$ such that
$h_j\neq 0$, then we have
$$v((\sum_{j\geq1} h_jp^{j-1}-1)!) \geq v(h_1!)+v((h_2p)!)
 +\sum_{j\geq 3}v((h_jp^{j-1}-1)!)\geq \sum ((j-1)h_j+v(h_j!)).$$
That is, $\beta \geq 0$. Hence $v(\tau_u)\geq m-1\geq s(p-2)-1$ by
(\ref{eqn 4}).

If for all $j\geq 3$, $h_j=0$, since $m=s(p-1)$, we have
$$n=h_1(p-1)+h_2(p^2-1)=mp+i(p-1)=s(p-1)p+i(p-1).$$
This implies that $h_1+h_2+h_2p=sp+i$. Since $d=h_1+h_2=i-kp$,
$h_2p=sp+(i-h_1-h_2)=sp+kp$. Then $h_2=s+k\geq s.$ On the other
hand,
\begin{align*}
   v(\tau_u)&=v(((h_1p+h_2p^2)-2)!)-(h_1+2h_2+v(h_1!)+v(h_2!)\\
    &=h_1+h_2p-1+v((h_1+h_2p-1)!)-(h_1+2h_2+v(h_1!)+v(h_2!))\\
    &\geq h_2(p-2)-1+v((h_2p-1)!)-v(h_2!)\\
    &\geq s(p-2)-1
\end{align*}
as desired.

 {\sc Case 2.} If there is exactly one
part $g\in\mathbf{Z}^+$ such that $g\neq p^\alpha-1$, $u_g=1$, and
if $t\neq g$ and $u_t\neq 0$ then $t=p^\alpha-1$. In this case, we
have $d=\sum_{j\ge1} h_j+1$, $n=\sum_{j\ge1} h_j(p^j-1)+g=(m+i)p-i,$
and
$$n+d-2=\sum_{j\ge1} h_jp^j+g-1=(m+i)p-i+d-2.$$
Then $g=(m+i)p-i+d-1-\sum_{j\ge1} h_jp^j$.

Also we can check that
\begin{align}
\label{-001} v(\gamma_u)=\sum_{j\ge1}(jh_j+v(h_j!))+v(g+1).
\end{align}
Let $\delta=i+1-d=i-\sum_{j\ge1} h_j$. Then $\delta \geq 0$ since
$d\leq i+1$. We have
$$\delta+g=(i+1-d)+(m+i)p-i+d-1-\sum_{j\ge1} h_jp^j=k'p,$$
where $k'=m+i-\sum h_jp^{j-1}.$ Then
$\sum_{j\ge1}h_jp^{j-1}=m+i-k'$. Since $g>0$ and $\delta \geq 0$, we
have $k'>0$. Furthermore, if $\delta =0$, then $g=k'p$.

On the other hand, since $n=(m+i)p-i$,
\begin{align}
\label{-002}n+d-2=(m+i)p-(\delta+1)=(m+i-k')p+g-1.
\end{align}
Hence by (\ref{eqn 0}), (\ref{eqn 1}), (\ref{-001}), (\ref{-002})
and Lemma \ref{A3 lemma 2.2} we have
\begin{align*}
    v(\tau_u)&= v(((m+i-k')p+g-1)!)-(\sum_{j\ge1} (jh_j+v(h_j!))+v(g+1))\\
    &\geq v(((m+i-k')p)!)+v((g-1)!)-\sum_{j\ge1} ((j-1)h_j+v(h_j!))-\sum_{j\ge1} h_j-v(g+1)\\
    &\geq m+i-k'-\sum_{j\ge1} h_j+v((m+i-k')!)+v((g-1)!) -v((\sum_{j\ge1} h_jp^{j-1})!)-v(g+1)\\
    &\geq m+i-k'-\sum_{j\ge1} h_j+v((g-1)!)-v(g+1)\\
    &\geq m-1+\delta'\\
    &\geq s(p-2)-1+\delta',
\end{align*}
where $\delta'=\delta-k'+1+v((g-1)!)-v(g+1)$. Thus it is sufficient
to show that $\delta'\ge0$.

If $\delta-k'\geq 0$, then $\delta'\ge v((g-1)!)-v(g+1)+1\ge0$ by
Lemma \ref{A3 lemma 2.3}.

It remains to consider the case $\delta-k'< 0$. In this case, let
$\delta =xp+r$ with $0\leq r<p$. Then $x=\lfloor \delta /p\rfloor$.
Clearly $x<\delta$ unless $\delta =0=r=x$, and $x<\delta -1$ if
$\delta \geq2$ since $p\geq3$.

Since $0<g=k'p-\delta=(k'-x)p-r$, we have $k'-x\geq1$. But
$$g-1=(k'-x)p-(r+1)=(k'-x-1)p+(p-r-1).$$
So $$v((g-1)!)=v(((k'-x-1)p)!)=k'-x-1+v((k'-x-1)!).$$ Hence
\begin{align*}
  \delta'&=\delta-k'+1+k'-x-1+v((k'-x-1)!)-v(g+1)\\
         &=\delta-x+v((k'-x-1)!)-v(g+1).
\end{align*}

If $p\nmid g+1$, then $\Delta'=\delta-x\ge0.$ Let now $p|g+1$. Since
$g+1=(k'-x)p-(r-1)$, we have $r=1$ and $g+1=(k'-x)p$. So
$$\delta'=\delta-x+v((k'-x-1)!)-v(k'-x)-1.$$

Since $r\neq 0$, we have $\delta \neq0$. So $\delta-x\geq 1$.
By Lemma \ref{A3 lemma 2.3} it
remains to show that $\delta -x=1$ is
impossible. If $\delta -x=1$, then $\delta=1$ and $x=0$, it is
sufficient to show that $k'-x\neq p$. But if $k'=p$ then $\delta
+g=1+g=k'p=p^2$, so $g=p^2-1$, which is impossible since $g\ne
p^\alpha-1$. The proof of Lemma \ref{lem -2} is
complete.\hfill$\Box$

\begin{cor}\label{cor 1}
Let $n=(m+i)p-i$ with $m=s(p-1)$. Assume that $w=n$ and $d(u)\leq
i+1$. Then $v(\tau_u) \geq s(p-2)-1$.
\end{cor}
\begin{proof}
This corollary follows immediately from Lemmas \ref{lem -1} and
\ref{lem -2}.
\end{proof}

Let's now recall the critical bounds for $\hat{{B}}_n/n$.

\begin{lem}\cite{[AHR]} \label{AHR thm 2.1}
Let $p$ be an odd prime. Suppose that $\omega(u)=n$ and $n=(p-1)s_0$
and $u_{p-1}<s_0$. Let $e=v(\gamma_u)-v((pu_{p-1})!)$ and
$\dot{n}=n-(p-1)u_{p-1}$. Then $n+d-2\geq ((u_{p-1}+e+1)p$ except
for the following cases where
$(u_{p-1}+e+1)p>n+d-2\geq(u_{p-1}+e)p$:

{\rm (i)} $p\geq 3,u_{p-1}=s_0-1$, and for $1\leq i\leq
\frac{p-3}{2}$, we have $u_i=u_{p-1-i}$, or if $i=\frac{p-1}{2}$, we
have $u_i=2$. In these cases, $\dot{n}=p-1$ and $e=0$;

{\rm (ii)} $p=3,u_2=s_0-4,u_8=1$. In this case, $\dot{n}=8$ and
$e=2$.
\end{lem}

For any positive integer $x$, by Lemma \ref{A3 lemma 2.1} we have
$$\frac{(xp)!}{x!p^x}-(-1)^x\equiv 0 \pmod{p^{v_p(x)+1}}.$$
Then associated to the prime number $p$, we can define an
arithmetical function $k_p$ for any positive integer $x$ by
$$k_p(x):=\frac{\frac{(xp)!}{x!p^x}-(-1)^x}{p^{v(x)+1}}.$$
Clearly $k_p(x)\in \mathbf{Z}$ for any positive integer $x$ and
$k_p(1)$ is the usual {\it Wilson quotient} (see, for instance,
\cite{[L]}). We call $k_p(x)$ {\it generalized Wilson quotient}. On
the other hand, for any positive integer $x$, we define
$$h_p(x):=(-1)^{x-1}\cdot\frac{(xp-2)!}{p^xx!}.$$
Then $h_p(x)=\displaystyle \frac{(-1)^{x-1}}{px}\cdot \prod_{i=1,
p\nmid i}^{xp-2}i$. We have the following result.

\begin{lem} \label{Lemma 222} Let $s$ and $l$ be positive integers
such that $v_p(l)=N$. If $v_p(s)<N$, then we have
$$h_p(l+s)-h_p(s)\equiv g_p(l,s)\pmod{p^{N+1}},
$$
where $g_p(l,s)$ is defined by
$$g_p(l,s):=\frac{1}{ps(l+s)}\bigg((-1)^lsk_p(l)
p^{N+1}-(-1)^slk_p(s)p^{v_p(s)+1}-l\bigg)\pmod{p^{N+1}}.$$
\end{lem}

\begin{proof}
First since $v_p(s)<N$, we have
\begin{align}\nonumber
\prod_{i=1,\,p\nmid i}^{(l+s)p-2}i &=\prod_{i=1,\,p\nmid
i}^{lp-1}i\cdot\prod_{i=lp+1,\,p\nmid
i}^{(l+s)p-2}i\\
&=\prod_{i=1,\,p\nmid
i}^{lp-1}i\cdot\prod_{i=1,\,p\nmid i}^{sp-2}(lp+i)\nonumber\\
&=\prod_{i=1,\,p\nmid i}^{lp-1}i\cdot\bigg(\prod_{i=1,\,p\nmid
i}^{sp-2}i+lp\cdot\prod_{i=1,\,p\nmid
i}^{sp-2}i\cdot\sum_{j=1,\,p\nmid j}^{sp-2} \frac{1}{j}+(lp)^2\cdot
k_1\bigg)\nonumber\\
&=\prod_{i=1,\,p\nmid i}^{lp-1}i\cdot\bigg(\prod_{i=1,\,p\nmid
i}^{sp-2}i\cdot\bigg(1+lp+lp\cdot\frac{1}{2}\cdot\sum_{j=2,\,p\nmid
j}^{sp-2} (\frac{1}{j}+\frac{1}{sp-j})\bigg)+(lp)^2\cdot k_1\bigg)
\nonumber
\end{align}
\begin{align}
&=\prod_{i=1,\,p\nmid i}^{lp-1}i\cdot\bigg(\prod_{i=1,\,p\nmid
i}^{sp-2}i\cdot\bigg(1+lp+lp\cdot\frac{sp}{2}\cdot\sum^{sp-2}_{j=2,\,p\nmid
j}
\frac{1}{j(sp-j)}\bigg)+(lp)^2\cdot k_1\bigg)\nonumber\\
&\equiv(1+lp)\cdot\prod_{i=1,\,p\nmid
i}^{lp-1}i\cdot\prod_{i=1,\,p\nmid i}^{sp-2}i
\hspace{30pt}\pmod{p^{v_p(s)+N+2}},\label{gongshi 1}
\end{align}
where $k_1$ is some positive integer.

Consequently, by the definition of $k_p(s)$ we can easily show that
\begin{align}\label{gongshi 2}
\prod_{i=1,\,p\nmid
i}^{lp-1}i=\frac{(lp)!}{l!p^l}=(-1)^l+k_p(l)p^{N+1}
\end{align}and
\begin{align}\label{gongshi 3}
\prod_{i=1,\,p\nmid
i}^{sp-2}i=(-1)\cdot\frac{1}{1-sp}\cdot\frac{(sp)!}{s!p^s}
=(-\sum_{j=0}^{\infty}(sp)^j)\cdot((-1)^s+k_p(s)p^{v_p(s)+1}).
\end{align}
Then by (\ref{gongshi 1}), (\ref{gongshi 2}) and (\ref{gongshi 3})
we can deduce that

\begin{align*}
h_p(l+s)-h_p(s)&=\frac{(-1)^{l+s-1}}{p(l+s)}\cdot\prod_{i=1,\,p\nmid
i}^{(l+s)p-2}i-\frac{(-1)^{s-1}}{p\,s}\cdot\prod_{i=1,\,p\nmid
i}^{sp-2}i\\
&\equiv\frac{(-1)^{l+s-1}}{p(l+s)}\cdot(1+lp)\cdot\prod_{i=1,\,p\nmid
i}^{lp-1}i\cdot\prod_{i=1,\,p\nmid
i}^{sp-2}i-\frac{(-1)^{s-1}}{p\,s}\cdot\prod_{i=1,\,p\nmid
i}^{sp-2}i \\
&=\bigg(\frac{(-1)^{s-1}}{p\,s}\cdot\prod_{i=1,\,p\nmid
i}^{sp-2}i\bigg)\bigg(\frac{(-1)^ls(1+lp)}{l+s}\cdot\prod_{i=1,\,p\nmid
i}^{lp-1}i-1\bigg)\\
&=\bigg(\frac{(-1)^{s-1}}{p\,s}\cdot\prod_{i=1,\,p\nmid
i}^{sp-2}i\bigg)\bigg(\frac{(-1)^ls(1+lp)}{l+s}
\cdot((-1)^l+k_p(l)p^{N+1})-1\bigg)
\\
 &\equiv\bigg(\frac{(-1)^{s-1}}{p\,s}\cdot\prod_{i=1,\,p\nmid
i}^{sp-2}i\bigg)\bigg(\frac{s}{l+s}((-1)^lk_p(l)p^{N+1}+lp)-
\frac{l}{l+s}\bigg)\\
&\equiv\bigg((-1)^{s-1}\cdot\prod_{i=1,\,p\nmid
i}^{sp-2}i\bigg)\bigg(\frac{(-1)^lk_p(l)p^{N}+l}{l+s}-
\frac{l}{ps(l+s)}\bigg)\\
&=\bigg((-1)^{s}\cdot\bigg(\sum_{j=0}^{\infty}(sp)^j\bigg)\cdot((-1)^s+
k_p(s)p^{v_p(s)+1})\bigg)\bigg(\frac{(-1)^lk_p(l)p^{N}+l}{l+s}-
\frac{l}{ps(l+s)}\bigg)\\
&\equiv(1+sp)(1+(-1)^sk_p(s)p^{v_p(s)+1})
\bigg(\frac{(-1)^lk_p(l)p^{N}+l}{l+s}-
\frac{l}{ps(l+s)}\bigg)\\
&\equiv g_p(l, s)\pmod{p^{N+1}}
\end{align*}
as required. Thus Lemma \ref{Lemma 222} is proved.
\end{proof}

Using the generalized Wilson quotient, we can now establish the
universal Kummer congruences modulo powers of odd primes. This is the
first main result of this paper.

\begin{thm}\label{thm 0}
Let $n=m+l(p-1)$ and $p^N\mid l$. Suppose that $m=s(p-1)$ and $s\geq
\lceil \frac{N+2}{p-2}\rceil$.
\begin{enumerate}
\item[{\rm (i)}] If $p\geq 5$, then
$$\frac{\hat{B}_n}{n}\equiv c_{p-1}^l\frac{\hat{B}_m}{m}
+g_p(l,s)c_{p-1}^{l+s}
 \pmod{{p^{N+1}\mathbf{Z}_p[c_1,\ldots ,c_n]}}.$$
\item[{\rm (ii)}] If $p=3$, then
$$\frac{\hat{B}_n}{n}\equiv c_2^l\frac{\hat{B}_m}{m}+g_3(l,s)c_{2}^{l+s} +\Psi \cdot
c_2^{l+s-4}c_8\pmod {3^{N+1}\mathbf{Z}_3[c_1,\ldots ,c_n]},$$
\end{enumerate}
where $\Psi=0$ if $s\equiv1\pmod 3$, $\Psi=-l$ if $s\equiv0\pmod 3$,
and $\Psi=l$ if $s\equiv-1\pmod 3$.

\end{thm}
\noindent{\it Proof.} First consider the terms $\tau_uc^u$ of
$\hat{B}_n/n$ where $u_{p-1}\geq l$. Let $u'_{p-1}=u_{p-1}-l:=q$ and
$u'_i=u_i$ for $i\neq p-1$. Then $w(u')=m$ and $c^{u'}\cdot
c_{p-1}^l=c^u$. It follows from (\ref{1}) that
$$\frac{\hat{B}_n}{n}-c_{p-1}^l\frac{\hat{B}_m}{m}=\sum_{w(u)=n, u_{p-1}\ge
l}(\tau _u-\tau _{u'})c^u+\sum_{w(u)=n, u_{p-1}<l}\tau _uc^u.$$

If $u_{p-1}=n/(p-1)=l+s$, then $u'_{p-1}=s=m/(p-1)$. Hence by the
definition and Lemma \ref{Lemma 222} we get
$$
\tau_u-\tau_{u'}=h_p(l+s)-h_p(s)\equiv g_p(l,s)\pmod {p^{N+1}}.
$$

Next assume $u_{p-1}<l+s$. Now let $\dot{u}_{p-1}=0$ and
$\dot{u}_i=u_i$ for $i\neq p-1$.

Let $v(\gamma_{\dot{u}})=e$. Then $\gamma_{\dot{u}}=\varepsilon p^e$
where $p\nmid \varepsilon$. Since $d'=d(u')=d(u)-l$ and
$\dot{d}=d(\dot{u})=d'-q$, we have
$$n+d-2=lp+m+d'-2=lp+qp+\dot{n}+\dot{d}-2.$$
Also $\gamma_u=(l+q)!p^{l+q}\gamma_{\dot{u}}$ and
$\gamma_{u'}=q!p^q\gamma_{\dot{u}}$.

 If $e=0$, then $(\gamma_{\dot{u}},p)=1$. So applying Lemma \ref{A3 lemma 2.4} (ii) with
 $a=\dot{n}+\dot{d}-2$ gives us
$$\tau_u\equiv \tau_{u'}\pmod {p^{N+1}}.$$

If $e>0$, then by Lemma \ref{AHR thm 2.1} we have
$\dot{n}+\dot{d}-2\geq (e+1)p$ with the exception of case (ii) where
$\dot{n}=8$, $e=2$, and $u_2=l+s-4$, $u_8=1$. Hence by Lemma \ref{A3
lemma 2.4} (iv), $\tau_u\equiv \tau_{u'}\pmod {p^{N+1}}$ without the
exceptional case.

We now turn to the exceptional term, which occurs if and only if
$p=3$, $u_2=l+s-4$, $u_8=1$. In this case, $d=l+s-4+1=l+s-3$,
$d'=s-3$. Then
$$n+d-2=2(l+s)+l+s-3-2=3(l+s)-5,$$
and
$$n'+d'-2=m+d'-2=3s-5.$$
Also we have
$$\gamma_u=3^{l+s-4}\cdot9\cdot(l+s-4)!=3^{l+s-2}(l+s-4)!$$
and $$\gamma_{u'}=3^{s-4}\cdot9\cdot(s-4)!=3^{s-2}(s-4)!.$$ Hence by
Lemmas \ref{A3 lemma 2.1} and \ref{A3 lemma 2.4} (ii)
\begin{align*}
\tau_u&=(-1)^{d-1}\frac{(3(l+s)-5)!}{3^{l+s-2}(l+s-4)!}\\
&=(-1)^{d-1}\frac{(3(l+s)-5)!}{3^{l+s-2}(l+s-2)!}(l+s-3)(l+s-2)\\
&=(-1)^{d-1}\frac{(3(l+s)-5)!}{3^{l+s-2}(l+s-2)!}(l^2+l(2s-5)+(s-2)(s-3))\\
&\equiv
(-1)^{d-1}(-1)^l\frac{(3s-5)!}{3^{s-2}(s-2)!}l(2s-5)+(-1)^{d-1}(-1)^l\frac{(3s-5)!}{3^{s-2}(s-2)!}(s-2)(s-3)\\
&\equiv(-1)^{s-4}l\cdot\frac{(3s-5)!}{3^{s-2}(s-2)!}(2s-5)+\tau_{u'}\pmod{3^{N+1}}.
\end{align*}
But by Lemma \ref{A3 lemma 2.0},
\begin{align*}
  \frac{(3s-5)!}{3^{s-2}(s-2)!}(2s-5)&=\frac{(3s-6)!}{3^{s-2}(s-2)!}(3s-5)(2s-5)\\
  &\equiv (-1)^{s-2}(-1)(2s-2)\pmod 3.
\end{align*}
So $$\tau_u\equiv
(-1)^{s-4+s-2+1}(2s-2)l+\tau_{u'}=(-1)l(2s-2)+\tau_{u'}\pmod{3^{N+1}}.$$
Thus $\tau_u\equiv \tau_u'\pmod {3^{N+1}}$ if $s\equiv1\pmod 3$,
$\tau_u\equiv \tau_u'-l \pmod {3^{N+1}}$ if $s\equiv0\pmod 3$ and
$\tau_u\equiv \tau_{u'}+l\pmod {3^{N+1}}$ if $s\equiv-1\pmod 3$.

We can now deal with the terms where $u_{p-1}<l$. In what follows
we assume that $u_{p-1}=l-x$ with $x\geq 1$. To finish the proof, it
is sufficient to show that if $w(u)=n$ and $u_{p-1}<l$ then
$\tau_u\equiv0 \pmod {p^{N+1}}$ if $s\geq \lceil
\frac{N+2}{p-2}\rceil$, with the single exception where $p=3$ and
$u$ is given by $u_{p-1}=l+s-4$, $u_8=1$.

If $n+d-2\geq lp$, then there exists an integer $k$ such that
$(l+k)p\le n+d-2< (l+k+1)p $. With our usual notation,
$\dot{n}=n-(l-x)(p-1)=m+x(p-1)$. Let $e=v(\gamma_{\dot{u}})$. Since
$u_{p-1}=l-x<l$, we have
\begin{align}\label{eqn 5}
n+d-2\geq (l-x+e+1)p
\end{align}
with the only exceptional case (ii) of Lemma \ref{AHR thm 2.1} which
was previously considered.

On the other hand, we have
\begin{align*}
  v(\tau_u)&=v((n+d-2)!)-(l-x)-v((l-x)!)-e\\
  &\geq v(((l+k)p)!)-(l-x)-v((l-x)!)-e\\
  &=(l+k)+v((l+k)!)-(l-x)-v((l-x)!)-e\\
  &\geq k+x+v(l)-e.
\end{align*}
Suppose now that $v(\tau_u)<v(lp)$. Then $e>k+x-1$, i.e., $e\geq
k+x$. But by (\ref{eqn 5}) we have $$n+d-2\geq
(l-x+e+1)p\geq(l-x+k+x+1)p=(l+k+1)p$$ which contradicts
$(l+k+1)p>n+d-2$. Thus in this case we infer that $v(\tau_u)\ge
v(lp)\ge N+1$.

Now it remains to consider the case $n+d-2<lp$. In this case, we
have $n+d=m+l(p-1)+d<lp+2$, i.e., $m+d\leq l+1$. Since
$u_{p-1}=l-x<d\leq l+1-m$, $x\geq m$. One may let $x=m+i$ with
$i\geq 0$. Then
$$w(\dot{u})=n-u_{p-1}(p-1)=m+x(p-1)=(m+i)p-i$$ and
$$d(\dot{u})=d-(l-x)\leq x+1-m=i+1.$$
But $u_{p-1}=l-x$. So $n+d-2\geq (l-x)p+\dot{n}+\dot{d}-2$ and
$v(\tau_u)\geq v(\tau_{\dot{u}})$. Since $s\geq \lceil
\frac{N+2}{p-2}\rceil$, then by assumption and replacing $n$ and $u$
by $\dot{n}=(m+i)p-i$ and $\dot{u}$, respectively, Corollary
\ref{cor 1} gives $\tau_u\equiv 0 \pmod {p^{N+1}}$.

The proof of Theorem \ref{thm 0} is complete. \hfill$\Box$

\section{Auxiliary results for powers of 2}

In the current section, we deal with the even prime $p=2$ case.
In order to establish the universal Kummer congruences modulo powers of 2,
we provide a detailed 2-adic analysis to many kinds of factorials and double
factorials. First we give several congruences modulo powers of 2 about double
factorials.

\begin{lem}\cite{[AHR]}\label{AHR Lem 4.1} Each of the following is true:

{\rm (i)} If $k$ is odd, then
$(2k-1)!!\equiv(-1)^{\frac{k-1}{2}}\pmod 4$.

{\rm (ii)} If $k\ge 1$, then $(4k-3)!!\equiv (-1)^{k-1}\pmod
{16}$.

{\rm (iii)} If $k\ge 1$ and $N\ge 3$, then $(k2^N-3)!!\equiv
-1\pmod {2^{N+1}}$.
\end{lem}

\begin{lem}\label{lemma 1}
Let  $k\ge 1$,  $N\ge 3$ and $a\ge 2$ be integers. Then
\begin{enumerate}
\item[(i)] $\frac{(k2^N+2a-3)!!}{(k2^N+1)!!}\equiv \left\{
\begin{array}{ll}(2a-3)!!\ \ \ \pmod{2^{N+1+\min\{v(a),\,N-1\}}}, &~{\rm if}~ 2\mid a,\\
(2a-3)!!+k2^N \pmod{2^{N+1}}, &~{\rm if }~ 2\nmid a.
\end{array}
\right.$ \item[(ii)] $(k2^N+2a-3)!!\equiv \left\{
\begin{array}{lll}
(2a-3)!!&\pmod {2^{N+1}},&{\rm if}~2\mid a,\\
(2a-3)!!+k2^N&\pmod {2^{N+1}},&{\rm if}~2\nmid a.
\end{array}
\right. $
\item[(iii)] If $k$ is odd, we have
$$
(k2^N-3)!!\equiv\left\{
\begin{array}{lll}
-1+(-1)^{\frac{k-1}{2}}\cdot 2^{N+1}
 &\pmod{2^{N+3}},&{\rm if}~N=3,\\
-1+(-1)^{\frac{k+1}{2}}\cdot 2^{N+1}
 &\pmod{2^{N+3}},&{\rm if}~N\ge4.
\end{array}
\right.
$$
In particular, we have $(k2^N-3)!!\equiv-1+2^{N+1}\pmod
{2^{N+2}}.$
\end{enumerate}
\end{lem}
\begin{proof}
(i). First we have
\begin{align}\label{equantion 1}
\frac{(k2^N+2a-3)!!}{(k2^N+1)!!}&=(k2^N+3)\cdot(k2^N+5)\ldots (k2^N+2a-3)\\
&=(2a-3)!!+k2^N\cdot\sum_{j\in S}\frac{(2a-3)!!}{j}+x\cdot
(k2^N)^2\nonumber\\
&\equiv(2a-3)!!+k2^N\cdot\sum_{j\in
S}\frac{(2a-3)!!}{j}\pmod{2^{2N}},\nonumber
\end{align}
where $x\in \mathbf{N}$ and $S=\{3,5,\ldots ,2a-3\}$. If $j\in S$,
then $2a-j\in S$. Clearly $j\neq 2a-j$ except that $j=a\in S$. If
$2|a$, then $a\not\in S$. Therefore we have
\begin{align*}
\sum_{j\in S}\frac{(2a-3)!!}{j}&=\sum_{j\in
S,j<a}\bigg(\frac{(2a-3)!!}{2a-j}+\frac{(2a-3)!!}{j}\bigg)\\
&=2a\sum_{j\in S,j<a}\frac{(2a-3)!!}{j(2a-j)}.\nonumber
\end{align*}
Then $v(k2^N\cdot\sum_{j\in S}\frac{(2a-3)!!}{j})\geq N+1+v(a)$.
Hence by (\ref{equantion 1}),  Lemma \ref{lemma 1} (i) is true for
the case $2|a$.

Now we consider the case  $2\nmid a$, then $a\in S$. Hence
\begin{align*}
\sum_{j\in S}\frac{(2a-3)!!}{j}&=\sum_{j\in
S,j<a}\bigg(\frac{(2a-3)!!}{2a-j}+\frac{(2a-3)!!}{j}\bigg)
+\frac{(2a-3)!!}{a}\\
&=2a\sum_{j\in
S,j<a}\frac{(2a-3)!!}{j(2a-j)}+\frac{(2a-3)!!}{a}\nonumber.
\end{align*}
Then $v(\sum_{j\in S}\frac{(2a-3)!!}{j})=0$ since
$v(\frac{(2a-3)!!}{a})=0$. Thus by (\ref{equantion 1}) we infer
that  Lemma \ref{lemma 1} (i) holds for the case $2\nmid a$.
Part (i) is proved.

(ii). Since $N\ge 3$, $2N\ge N+1$. Then by Lemma \ref{AHR Lem
4.1} (iii)
\begin{align*}(k\cdot 2^N+1)!!&=(k\cdot
2^N-3)!!\cdot(k\cdot 2^N-1)\cdot (k\cdot 2^N+1)\\
&\equiv -(k\cdot 2^N-3)!!\\
&\equiv1\pmod{2^{N+1}}.
\end{align*}
So the desired result follows immediately from part (i).

(iii).  If $k\equiv1\pmod 4$, then by part (ii)
\begin{align}\label{001}
(k2^N-3)!!=(t_1\cdot2^{N+2}+2^N-3)!!\equiv(2^N-3)!!\pmod{2^{N+3}},
\end{align}
where $t_1\ge 0$ is an integer. If $k\equiv3\pmod 4$, then by part
(ii)
\begin{align}\label{002}
(k2^N-3)!!=(t_2\cdot2^{N+2}+3\cdot2^N-3)!!\equiv(3\cdot2^N-3)!!\pmod{2^{N+3}},
\end{align}
where $t_2\ge 0$ is an integer. By part (i) we have
\begin{align}\label{003}
(3\cdot2^N-3)!!&\equiv\frac{(2^{N+1}+2^N-3)!!}{(2^{N+1}+1)!!}
\cdot(-1)\cdot\frac{(2^{N}+2^N-3)!!}{(2^{N}+1)!!}\cdot(-1)\cdot(2^{N}-3)!!\\
&\equiv((2^{N}-3)!!)^3\pmod{2^{2N}}.\nonumber
\end{align}
Since $N\ge 3$, $2N\ge N+3$. Then by (\ref{002}) and (\ref{003}) we
get
\begin{align}\label{004}
(k2^N-3)!!\equiv((2^N-3)!!)^3\pmod{2^{N+3}}.
\end{align}

First consider the case $N=3$. Clearly $(2^3-3)!!\equiv15\pmod
{2^6}$. Hence by (\ref{001}) we have $ (k\cdot2^3-3)!!\equiv
15=-1+(-1)^{\frac{k-1}{2}}\cdot 2^{4}\pmod{2^{6}}$ if $k\equiv 1
\pmod 4$. By (\ref{004}), $(k\cdot2^3-3)!!\equiv15^3\equiv -17=
-1+(-1)^{\frac{k-1}{2}}\cdot 2^{4}\pmod{2^{6}}$ if $k\equiv 3
\pmod 4$. Part (iii) is proved for the $N=3$ case. In what follows
we deal with the $N\ge 4$ case.

We claim that $(2^{N}-3)!!\equiv-1-2^{N+1} \pmod {2^{N+3}}$ for
$N\ge4$. We use induction on $N$. Evidently
$(2^4-3)!!\equiv-1-2^5\pmod {2^7}$ if $N=4$. Assume that
$(2^N-3)!!\equiv-1-2^{N+1}\pmod {2^{N+3}}$ for some $N\ge 4$. Since
$N\ge4$, $2N\geq N+4$. Then by part (i) and the induction hypothesis, we
get
\begin{align*}
(2^{N+1}-3)!!&\equiv\frac{(2^{N}+2^N-3)!!}{(2^{N}+1)!!}\cdot(-1)\cdot(2^{N}-3)!!\\
&\equiv-((2^N-3)!!)^2\\
&\equiv-1-2^{N+2} \pmod {2^{N+4}}.
\end{align*}
Hence the claim is true.  So the desired result follows
immediately from the claim and (\ref{001}), (\ref{004}). Part
(iii) is proved.
\end{proof}

\begin{lem}\label{lemma 3}
Let $a\ge0$ and $i\ge1$ be integers. Define
$f_a(i,j):=\frac{(a+1)\ldots (a+2i)}{a+j}$ for $j\in\{1,\ldots ,2i\}$.
Then
\begin{align}
v(\sum_{j=1}^{2i}f_a(i,j))\ge \left\{
\begin{array}{ll}
i-1, &{ if }~1\le i\le 3\\
i, &{\it otherwise}.\\
\end{array}
 \right.
\end{align}
\end{lem}
\begin{proof}
First we have $v(f_a(1,1)+f_a(1,2))=v(2a+3)=0$. So Lemma \ref{lemma
3} is true for $i=1$. For $i=2$, we can easily check that
$\sum_{j=1}^4f_a(2,j)\equiv2a(a+1)+2\equiv2\pmod 4$. Hence
$v(\sum_{j=1}^4f_a(2,j))=1$ as required. Lemma \ref{lemma 3} is true
when $i=2$. For $i=3$, by the fact that $8$ divides the product of
any four consecutive integers, we get $8|f_a(3,j)$ for $j=1,2,5,6$.
On the other hand, we have  $f_a(3,3)\equiv a^3(a+1)^2\equiv 0\pmod
4$ and $f_a(3,4)\equiv a^2(a+1)^2(a+3)\equiv 0\pmod 4$. It follows
that
$$v(\sum_{j=1}^6f_a(3,j))\ge \min_{1\le j\le 6} v(f_a(3,j))\ge
2$$
as required.

Now let $i\ge4$. We may let $a=2b+b'$, where $b\in \mathbf{Z}_{\ge
0}$ and $b'=0$ or $1$. Then we get
\begin{align}\label{equantion 4}
v(f_a(i,j))=v\Big(\frac{(2b+2)\ldots (2b+2i)}{2b+j+b'}\Big)
=i+v\Big(\frac{(b+1)\ldots (b+i)}{2b+j+b'}\Big).
\end{align}
If $2\nmid (b'+j)$, we have
$v\Big(\frac{(b+1)\ldots (b+i)}{2b+j+b'}\Big)\ge 3.$ If $2|(b'+j)$,
then $b+(j+b')/2\in\{b+1,\ldots , b+i\}$. Thus we have
$v\Big(\frac{(b+1)\ldots (b+i)}{2(b+(j+b')/2)}\Big)\ge 0$ since
$i\ge4$. Hence by (\ref{equantion 4}), we have $v(f_a(i,j))\ge i$
for all $1\le j\le 2i$. Therefore
$$v(\sum_{j=1}^{2i}f_a(i,j))\ge \min_{1\le j\le
2i}(v(f_a(i,j)))\ge i.$$ The proof of Lemma \ref{lemma 3} is
complete.
\end{proof}

\noindent{\bf Remark.}
 In fact, we infer that
$v\Big(\frac{(b+1)\ldots (b+i)}{2b+j+b'}\Big)\ge 1$ for $i\ge6$ and
$v\Big(\frac{(b+1)\ldots (b+i)}{2b+j+b'}\Big)\ge 3$ for $i\ge8$. Then
$v(f_a(i,j))\ge i+1$ for $i\ge6$  and $v(f_a(i,j))\ge i+3$ for
$i\ge8$ by (\ref{equantion 4}).

\begin{lem}\label{lemma 4}
Let $q,r,a\in \mathbf{N}$, $e\in \mathbf{Z}^+$ and $v(l)=N$.  Define
$\delta_r=l$ for $r=1,2$ and $\delta_r=0$ for $r\neq1,2$. Then
\begin{enumerate}
\item[(i)] $\frac{(l+q+2r)!}{(l+q)!r!}\equiv\frac{(q+2r)!}{q!r!}+\delta_r\pmod
{2^{N+1}}$.
\item[(ii)] If $a=0$ or $1$, $\frac{(2(l+q+2r)+a)!}{2^{l+q+2r}(l+q)!r!}\equiv
\frac{(2(q+2r)+a)!}{2^{q+2r}q!r!}+\delta_r\pmod {2^{N+1}}$.
\item[(iii)] If $a\ge 2e$,  $\frac{(2(l+q+2r)+a)!}{2^{l+q+2r}(l+q)!r!}\equiv
\frac{(2(q+2r)+a)!}{2^{q+2r}q!r!}\pmod {2^{N+e}}$.

\item[(iv)] If $a\ge 2(e+1)$,
$\frac{(2(l+q+2r)+a)!}{2^{l+q+2r+e}(l+q)!r!}\equiv
\frac{(2(q+2r)+a)!}{2^{q+2r+e}q!r!}\pmod {2^{N+1}}$.
\end{enumerate}
\end{lem}

\begin{proof}
(i). For $r=0$, the congruence is trivial. Now let $r>0$. If $r\ge
N+1$, we have
$$v\bigg(\frac{(l+q+2r)!}{(l+q)!r!}\bigg)\ge v((2r)!)-v(r!)=r\ge N+1$$
since $v((l+q+2r)!)\ge v((l+q)!)+v((2r)!)$ and $v((2r)!)=r+v(r!)$.
Similarly we have $v(\frac{(q+2r)!}{q!r!})\ge N+1$. Then
$\frac{(l+q+2r)!}{(l+q)!r!}\equiv\frac{(q+2r)!}{q!r!}\equiv 0\pmod
{2^{N+1}}$ as required. If $r\le N$, then $v(r!)\le r-1\le N-1$. On
the other hand, we have
\begin{align}\label{005}
\frac{(l+q+2r)!}{(l+q)!r!}&=\frac{(l+q+1)\ldots (l+q+2r)}{r!}\\
&\equiv\frac{(q+1)\ldots (q+2r)}{r!}+\frac{l}{r!}\sum_{j=1}^{2r}\frac{(q+1)\ldots (q+2r)}{q+j}\nonumber\\
&\equiv\frac{(q+2r)!}{q!r!}+\frac{l}{r!}\sum_{j=1}^{2r}f_q(r,j)\nonumber\pmod
{2^{N+1}},
\end{align}
where $f_a(i,j)$ is defined as in Lemma \ref{lemma 3}. From
(\ref{005}) and Lemma \ref{lemma 3} we deduce that if $r=1,2$, then
\begin{align*}
\frac{(l+q+2r)!}{(l+q)!r!}\equiv\frac{(q+2r)!}{q!r!}+l\pmod
{2^{N+1}},
\end{align*}
and if $r\ge 3$, then
\begin{align*}
\frac{(l+q+2r)!}{(l+q)!r!}\equiv\frac{(q+2r)!}{q!r!}\pmod {2^{N+1}}.
\end{align*}
So part (i) is proved.

(ii). Since $v(2l)=N+1$, we have
\begin{align}\label{equantion 7}
(2l+2q+4r\pm1)!!\equiv(2q+4r\pm1)!!\pmod {2^{N+1}}
\end{align}
by Lemma \ref{lemma 1} (ii). Then multiplying congruence (i) by
(\ref{equantion 7}) and noting that $(2q+4r\pm1)!!\equiv 1\pmod 2$
the desired result follows.

(iii). To deduce part (iii), let $S=\{2,4,\ldots ,2e\}$. Then
\begin{align*}
\prod_{x=1}^a(2(l+q+2r)+x)&=\prod_{i=1}^e2(l+q+2r+i)\cdot
\prod_{x\not\in S,x=1}^a (2(l+q+2r)+x)\\
&=2^e\prod_{i=1}^e(l+q+2r+i)\cdot \prod_{x\not\in S,x=1}^a (2l+2q+4r+x)
\end{align*}
Thus, using congruence (ii) $\pmod {2^{N}}$ for $a=0$, we get
\begin{align*}
&\frac{(2(l+q+2r)+a)!}{2^{l+q+2r}(l+q)!r!}
=\bigg(\prod_{x=1}^a(2(l+q+2r)+x)\bigg)\frac{(2(l+q+2r))!}{2^{l+q+2r}(l+q)!r!}\\
&\equiv 2^e\bigg(\prod_{i=1}^e(q+2r+i)\bigg)\bigg(\prod_{x\not \in
S,x=1}^a (2q+4r+x)\bigg)\frac{(2(q+2r))!}{2^{q+2r}q!r!}\\
&\equiv \frac{(2(2+2r)+a)!}{2^{q+2r}q!r!}\pmod {2^{N+e}}.
\end{align*}

(iv). To deduce part (iv), use the congruence (iii) with $e$
replaced by $e+1$ and then divide by $2^e$. Lemma 4.4 is proved.
\end{proof}

\begin{lem}\label{lemma add}
Let $n=l+m$, $N\ge 3$ and $l=k2^{N}$, $k\in \mathbf{Z}^+, 2\nmid k$.
Let $g(a):=(-1)^{a-1}\frac{(2a-3)!!}{2a}, $ and $a\in
\mathbf{Z}^+$. Then
\begin{align*}
g(n)\equiv g(m)+\left\{\begin{array}{lll}
(-1)^{\frac{m-1}{2}}\frac{l}{2}&\pmod {2^{N+1}},& if~2\nmid m,\\
l+\frac{l}{2mn}&\pmod {2^{N+1}},& if~2|m,4\nmid m,\\
l-\frac{l}{2mn}&\pmod {2^{N+1}},& if~4|m, 8\nmid m.
\end{array} \right.
\end{align*}
\end{lem}
\begin{proof} If $2\nmid m$, then applying Lemma
\ref{lemma 1} (ii), we obtain
$(2n-3)!!=(2l+2m-3)!!\equiv(2m-3)!!+2l\pmod{2^{N+2}}$. By Lemma
\ref{AHR Lem 4.1} (i), $(2m-3)!!=(2m-3)(2(m-2)-1)!!\equiv
(-1)^{\frac{m-1}{2}}\pmod4$. Since $2\nmid m$ and $n=l+m$,
$\frac{1}{mn}=\frac{1}{m^2+ml}\in 1+4\mathbf{Z}_2$.  Thus
\begin{align*}
g(n)-g(m)&\equiv(2m-3)!!(-\frac{l}{2mn})+\frac{l}{n}\\
&\equiv(-1)^{\frac{m-1}{2}}(-\frac{l}{2})+l\\
&\equiv(-1)^{\frac{m-1}{2}}\frac{l}{2}\pmod{2^{N+1}}.
\end{align*}

If $2|m$, then by Lemma \ref{lemma 1} (i),
$$(2n-3)!!=(2l+2m-3)!!\equiv-(2l-3)!!(2m-3)!!\pmod{2^{N+2+\min\{v(m),N\}}}$$
Thus for $v(m)<N$, we have
\begin{align}\label{equantion 100}
g(n)-g(m)\equiv
\bigg(\frac{(2l-3)!!}{2(l+m)}+\frac{1}{2m}\bigg)(2m-3)!!\pmod
{2^{N+1}}.
\end{align}

If $2|m$ and $4\nmid m$, then $(2m-3)!!\equiv 1\pmod {16}$ by Lemma
\ref{AHR Lem 4.1}(ii). Thus applying (\ref{equantion 100} ) and
Lemma \ref{lemma 1} (iii) gives us
\begin{align}\label{equantion 200}
g(n)-g(m)&\equiv\bigg(\frac{-1+2^{N+2}}{2(l+m)}+\frac{1}{2m}\bigg)(2m-3)!!
\nonumber\\
&\equiv\bigg(\frac{2^N}{\frac{l}{2}+\frac{m}{2}}+\frac{l}{2}
\cdot\frac{1}{m(m+l)}\bigg)\pmod{2^{N+1}}.
\end{align}
Since $\frac{l}{2}+\frac{m}{2}\in  1+2\mathbf{Z}_2$, we have
$g(n)-g(m)\equiv l+\frac{l}{2mn} \pmod {2^{N+1}}$ by (\ref{equantion
200}).

If $4|m$ and $8\nmid m$, then we may assume $m=4\lambda$ with
$2\nmid \lambda$. Hence by Lemma \ref{lemma 1} (iii),
\begin{align}\label{equantion 300}
(2m-3)!!\equiv -1+(-1)^{\frac{\lambda-1}{2}}\cdot2^4\equiv-1+4m\pmod
{2^6}.
\end{align}
Also by Lemma \ref{lemma 1} (iii) and noting that $N\ge3$, we have
\begin{align}\label{equantion 400}
(2l-3)!!=(k\cdot2^{N+1}-3)!!\equiv -1+(-1)^{\frac{k+1}{2}}\cdot
2^{N+2}\pmod {2^{N+4}}.
\end{align}
Thus using (\ref{equantion 100}), (\ref{equantion 300}) and
(\ref{equantion 400}), we get
\begin{align*}
g(n)-g(m)&\equiv\bigg(\frac{-1+(-1)^{\frac{k+1}{2}}\cdot
2^{N+2}}{2(l+m)}+\frac{1}{2m}\bigg)(2m-3)!!\\
&\equiv\bigg(\frac{(-1)^{\frac{k+1}{2}}2^{N+1}}{l+m}+\frac{l}{2m(l+m)}\bigg)(-1+4m)
\nonumber\\
&\equiv \frac{(-1)^{\frac{k-1}{2}}2^{N+1}+2l}{l+m}-\frac{l}{2mn}\\
&\equiv l-\frac{l}{2mn} \pmod{2^{N+1}}.
\end{align*}
The proof of Lemma \ref{lemma add} is complete.
\end{proof}

In the rest of this paper, we always let
$e:=v(\gamma_u)-v((2u_1)!)-(2u_3+v(u_3!))$ and define $\dot {u}$ by:
$\dot {u}_1=\dot {u}_3=0$ and $\dot {u}_i=u_i$ for $i\ne 1, 3$. Let
$\dot{n}=w(\dot {u})$. Then $\dot{n}=n-u_1-3u_3$ and $v(\gamma_{\dot
{u}})=e$.

\begin{lem}\label{thm 1}
Suppose $w(u)=n$.  Then $n+d-2=2(u_1+2u_3+e)+\Gamma$, where

{\rm (i)} $\Gamma=-2$ if $\dot{n}=0$,

{\rm (ii)} $\Gamma=0$ if
$\dot{n}=7\cdot2^\alpha$ and $u_7=2^\alpha$, $\alpha\in \mathbf{N}$,

{\rm (iii)} $\Gamma=1$ if $ \dot{n}=2$,

{\rm (iv)} $\Gamma\ge 2$ otherwise.
\end{lem}

\begin{proof}If $\dot{n}=0$ i.e., $n=u_1+3u_3$, we have $e=0$ and
$n+d-2=2u_1+4u_3-2$. Then $\Gamma=-2$.

Now we assume that $\dot{n}>0$. Let $\dot{u}_1=\dot{u}_3=0$ and
$\dot{u}_i=u_i \ (i\neq 1,3)$. Then $w(\dot{u})=\dot{n},
v(\gamma_{\dot{u}})=e$ and $n+d-2=2u_1+4u_3+\dot{n}+\dot{d}-2.$ So
replacing $u$ by $\dot{u}$, in what follows we can assume that
$u_1=u_3=0$. Note that $n+d=\sum_{i=1}^n(i+1)u_i$ and
$e=\sum_{i=1}^n e_i$, where
$$e_i:=v(i+1)u_i+v(u_i!).$$ It follows immediately that $e_i>0$ if and only if
either $2|(i+1)$ or $u_i\ge 2$.

First consider case (ii): $\dot{n}=7\cdot2^\alpha$ and
$u_7=2^\alpha$. Then
$e_i=3\cdot2^\alpha+v(2^\alpha!)=4\cdot2^\alpha-1.$ Hence we have
$(i+1)u_i=8u_7=2(e_7+1)$. Since $n+d-2=8u_7-2$ and
$2(u_1+2u_3+e)=2e_7$, we have $\Gamma=0$.

For case (iii): $\dot{n}=2$, $i=2,u_2=1$. Clearly $e_i=0$. Since
$u_1=u_3=e=0$, we have $n+d-2=(i+1)u_i-2=1=2(u_1+2u_3+e)+1$. Hence
$\Gamma=1$.

In what follows we deal with case (iv).  We claim that
\begin{align}\label{equantion 8}(i+1)u_i\ge
2(e_i+1)+2=2(v(i+1)u_i+v(u_i!)+1)+2
\end{align}
for all cases where $i\neq 1,3$ apart from  cases (ii) and (iii).

If  $2\nmid (i+1)$ and $u_i\ge 2,$ then $e_i= v(u_i!)\le u_i-1.$ We
can deduce that
$$(i+1)u_i-2\ge 3u_i-2\ge 2u_i\ge 2(e_i+1).$$

If $2|(i+1)$,  we may let $i+1=\varepsilon 2^t$ for some $t\in
\mathbf{N}$ with $2\nmid \varepsilon$. Then  $e_i=tu_i+v(u_i!)$,
$(i+1)u_i-2=\varepsilon 2^t u_i-2$ and $2(e_i+1)=2(t
u_i+v((u_i)!)+1)\le 2(t+1)u_i.$

For $t\ge 4$, we have $2^t>2(t+2)$. It implies that
$$(i+1)u_i-2\ge 2^t u_i-2 \ge 2(t+1)u_i\ge 2(e_i+1).
$$

For $t= 3$, we have $i+1= 2^3\varepsilon$. If $\varepsilon\ge 3$,
then $(i+1)u_i-2\ge 3\cdot 2^3 u_i-2\ge 2(3+1)u_i\ge2(e_i+1).$ If
$\varepsilon=1$, then $(i+1)u_i-2=2^3u_i-2=8u_i-2$ and
$e_i=3u_i+v(u_i!)=3u_i+u_i-s_2(u_i)$. Therefore $(i+1)u_i-2\ge
2(e_i+1)$ if $s_2(u_i)\ge2.$  But if $s_2(u_i)=1$, we have $i=7,
u_i=2^\alpha,$ which is case (ii).

For $t=1,2$, we have $i+1= 2\varepsilon$ or $2^2\varepsilon$. Since
$i\neq 1,3$, we have $\varepsilon \geq 3$. Then
$$
(i+1)u_i-2\ge 3\cdot 2^t u_i-2\ge 2(t+1)u_i.
$$
Hence the claim (\ref{equantion 8}) holds as specified.

Finally, if $u_i\ge 1$, then $(i+1)u_i\ge 2$, while if $(i+1)u_i\ge
2(e_i+1)$ and $(j+1)u_j\ge 2(e_j+1)$, then
\begin{align*}
(i+1)u_i+(j+1)u_j-2\ge 2(e_i+e_j+1).
\end{align*}
So we have
$n+d-2\ge 2(e+1)$ for case (iv), which follows by adding the
``local" inequalities for each value of $i$ separately.
\end{proof}

\begin{lem}
\label{thm 2} Let $w(u)=n$. Assume that $\dot{n}>0.$ Then
$v(\tau_u)\ge u_3+\lceil \frac{1}{2}\dot{n}\rceil-1$ except for
$\dot{n}=7u_7$ where $v(\tau_u)\ge u_3+\lceil
\frac{1}{2}\dot{n}\rceil-3$.
\end{lem}

\begin{proof}
 Let $\dot{u}_1=\dot{u}_3=0$ and
$\dot{u}_i=u_i(i\neq 1,3)$. Then $w(\dot{u})=\dot{n}$ and
$n+d-2=2u_1+4u_3+\dot{n}+\dot{d}-2$. So we have
\begin{align*}
v(\tau_u)&=v((n+d-2)!)-v(\gamma_u)\\
&=v((2u_1+4u_3+\dot{n}+\dot{d}-2)!)-(u_1+v(u_1!)
+2u_3+v(u_3!))-v(\gamma_{\dot{u}})\\
&\ge
v((4u_3)!)+v((\dot{n}+\dot{d}-2)!)-2u_3-v(u_3!)-v(\gamma_{\dot{u}})\\
&= u_3+v((\dot{n}+\dot{d}-2)!)-v(\gamma_{\dot{u}}).
\end{align*}
Noticing that $v(\gamma_{\dot{u}})=e=\sum e_i$, it suffices to prove
that $v((\dot{n}+\dot{d}-2)!)-\sum_{i=1}^n e_i\ge\lceil
\frac{1}{2}\dot{n}\rceil-1$. We claim that
\begin{align}\label{equantion 9}
v\big((i+1)u_i-2)!\big)-e_i\ge \Big\lceil\frac{1}{2}i
u_i\Big\rceil-1
\end{align}
for $u_i\ge1$ and $i\neq 1,3,7.$

If $2\nmid (i+1)$, then $e_i=v(i+1)u_i+v(u_i!)=v(u_i!)$ and
$\Big\lceil\frac{1}{2}i u_i\Big\rceil=\frac{i}{2} u_i$. Thus we get
$v((i+1)u_i-2)\geq v(u_i!)+v((iu_i-2)!)\ge
v(u_i!)+\frac{i}{2}u_i-1\ge e_i+\Big\lceil\frac{1}{2}i
u_i\Big\rceil-1$, as claimed.

If $2|(i+1)$, we may let $i+1=\delta 2^t$ where $t\in \mathbf{Z}^+$
and $2\nmid \delta$. Then $e_i=tu_i+v(u_i!)$, and
$\frac{i+1}{2}u_i=\delta 2^{t-1}u_i\ge \Big\lceil\frac{1}{2}i
u_i\Big\rceil$. Therefore we gain
\begin{align}\label{equantion 10}
v\big(((i+1)u_i-2)!\big)&=v\big((2(\delta 2^{t-1}u_i-1))!\big)\\
&=\delta 2^{t-1}u_i-1+v((\delta2^{t-1}u_i-1)!)\nonumber\\
&\ge\Big\lceil\frac{1}{2}i
u_i\Big\rceil-1+v((\delta2^{t-1}u_i-1)!).\nonumber
\end{align}
Let $\delta\ge 3$. Then we have $$v((\delta2^{t-1}u_i-1)!)\ge
v((3\cdot2^{t-1}u_i-1)!)\ge v((2tu_i)!)\ge tu_i+v((tu_i)!)\ge e_i
$$ since $3\cdot2^{t-1}u_i-1\ge 2tu_i$ for $t\ge 1$ and $u_i\ge1$.
So the claim holds for this case by (\ref{equantion 10}).

Now let $\delta=1$, i.e., $i+1=2^t$. Since $i\neq 1,3,7$, $t\geq 4$.
Then by Lemma \ref{A3 lemma 2.0} we have
$$v((2^{t-1}u_i-1)!)=3\cdot2^{t-3}u_i-2+v((2^{t-3}u_i-1)!)\ge tu_i+v(u_i!)\ge e_i$$
since $3\cdot2^{t-3}u_i-2\ge tu_i$ and $2^{t-3}u_i-1\ge u_i$ for
$t\ge 4$ and $u_i\ge1$. So the claim holds for this case by
(\ref{equantion 10}).

On the other hand, if $i\neq 1,3,7$, we have
$v\big(((i+1)u_i)!\big)\ge v\big(((i+1)u_i-2)!)\big)+1\ge
e_i+\Big\lceil\frac{1}{2}i u_i\Big\rceil$ by claim (\ref{equantion
9}). If $i=7$, by Lemma \ref{A3 lemma 2.0} we have
$v(8u_i)!)=7u_i+v(u_i!)\ge e_7+\Big\lceil\frac{7}{2}u_7\Big\rceil.$
Then for $i\neq 1,3$,
\begin{align}\label{equantion 11}v((i+1)u_i)!)\ge
e_i+\Big\lceil\frac{1}{2}i u_i\Big\rceil .
\end{align}

Moreover, if $u_j\ge1$ for some $j\neq 7$, then by (\ref{equantion
9}) and (\ref{equantion 11}) we have

\begin{align*}
v\big((\dot{n}+\dot{d}-2)!\big)&=v\big((\sum_{i=1}^n(i+1)u_i-2)!\big)\\
&=v\big((\sum_{i=1,i\neq j}^n(i+1)u_i+(j+1)u_j-2)!)\big)\\
&\ge \sum_{i=1}^n e_i+\Big\lceil\frac{\sum_{i=1}^n iu_i}{2}\Big\rceil-1,
\end{align*}
as required. Finally if $\dot{n}=7u_7$, we can compute $v(\tau_u)\ge
u_3+v((8u_7-2)!)-e_7\ge u_3+4u_7-3-v(u_7).$ We can easily show that
$4u_7-v(u_7)\ge\lceil \frac{7u_7}{2}\rceil=\lceil
\frac{1}{2}\dot{n}\rceil.$ Then $v(\tau_u)\ge u_3+\lceil
\frac{1}{2}\dot{n}\rceil-3.$ Lemma 4.7 is proved.
\end{proof}

\section{Universal Kummer congruences for powers of 2}

We are now in a position to give the universal Kummer congruences
modulo powers of 2 which consists of Theorems \ref{thm 3} and
\ref{thm 4}. This is the second main result of this paper.
We begin with Theorem \ref{thm 3}.

\begin{thm}\label{thm 3}
Let $n\ge 2$ be an even number.

{\rm (i)} If $v(n)=1$, then
\begin{align*}
\frac{\widehat{B}_n}{n}\equiv
-\frac{1}{2n}c_1^{n}+(\frac{n}{2}-1)c_1^{n-3}c_3+\frac{3(n-4)}{4}
c_1^{n-6}c_3^2-c_1^{n-2}c_2+2c_1^{n-5}(c_2c_3+c_1c_4)\\
\pmod{4\mathbf{Z}_2[c_1,\ldots ,c_n]}.
\end{align*}

{\rm (ii)} If $v(n)\geq 2$, then
\begin{align*}
\frac{\widehat{B}_n}{n}\equiv
&(\frac{1}{2n}-2)c_1^{n}-3\frac{n-2}{2}c_1^{n-3}c_3+\frac{n-4}{4}c_1^{n-6}c_3^2
+\frac{n-8}{4}c_1^{n-12}c_3^4\\
&-3c_1^{n-2}c_2+2c_1^{n-4}c_4+4c_1^{n-4}c_2^2+(n-4)c_1^{n-8}c_2c_3(c_3+c_1^{3})
\pmod{8\mathbf{Z}_2[c_1,\ldots ,c_n]}.
\end{align*}
\end{thm}

\begin{proof}
By Lemma \ref{thm 2}, and noting that $\dot{n}=\sum_{i\neq1,3}
iu_i$, we have only to consider the following cases:

(1) $\dot{n}=7$, $u_7=1$, $u_3\le1$. Then $u_1=n-7-3u_3$.

(2) $\dot{n}=6$, $u_1=n-6$. In this case, we have $u_2=3$, or
$u_2=u_4=1$, or $u_6=1$.

(3) $\dot{n}=5$, $u_5=1$ and $u_1=n-5$.

(4) $\dot{n}=4$, $u_1=n-4-3u_3$ and $u_3\leq 1$. In this case,
either $u_2=2$ or $u_4=1$.

(5) $\dot{n}=2$, $u_1=n-2-3u_3$, $u_2=1$ and  $u_3\leq 2$.

(6) $\dot{n}=0$, $u_3\le4$. In fact, if $u_3\geq 5$, we have
$v(\tau_u)\geq v((2u_3-1)!)-v(u_3!)-1\geq 3$ since $\dot{n}=0$. So
for our purpose, we can assume $u_3\le4$.

For cases (1)-(3) and (4) with $u_3=1$, by the definition of $\tau
_u$, we can easily check that $v(\tau_u)\geq 3$. We omit the details here.

For case (4) with $u_3=0$, if $u_2=2$, then
$v(\tau_u)=v((2n-4)!)-(n-4)-v((n-4)!)-v(2!)=1+v(n-2)=2+v(\frac{n}{2}-1)$,
i.e., $\tau_u \equiv 0\pmod{8}$ for $v(n)=1$ and $\tau_u \equiv
4\pmod{8}$ for $v(n)\geq 2$. If $u_4=1$, then
$\tau_u=(-1)^{n-4}\frac{(2n-5)!}{2^{n-4}(n-4)!5}=2\frac{(n-3)(2n-5)!!}{5}
\equiv (-1)^{\frac{n}{2}}2\pmod{8}.$

For case (5), if $u_3=2$, then
$v(\tau_u)=v((2n-7)!)-(n-4)-v((n-8)!)-v(2!)=v((n-4)(n-6))-1$. Hence
$\tau_u \equiv4\pmod{8}$ unless $n\equiv4$ or $-2\pmod{8}$ where
$\tau_u \equiv0\pmod{8}$; if $u_3=1$, then
$\tau_u=(-1)^{n-4}\frac{(2n-5)!}{2^{n-5}(n-5)!\cdot3\cdot4}
=\frac{(n-3)(n-4)(2n-5)!!}{3}$. It implies that
$\tau_u\equiv0\pmod{8}$ if $v(n)=2$, $\tau_u\equiv4\pmod{8}$ if
$v(n)>2$, and $\tau_u\equiv-n\pmod{8}$ if $v(n)=1$; if $u_3=0$, then
$\tau_u=(-1)^{n-2}\frac{(2n-3)!}{2^{n-2}(n-2)!\cdot3}=\frac{(2n-3)!!}{3}\equiv
3\cdot(-1)^{\frac{n}{2}-1}\pmod{8}.$

For case (6), if $u_3=4$, then
$\tau_u=(-1)^{n-9}\frac{(2n-10)!}{2^{n-12+8}(n-12)!\cdot4!}
=-\frac{(n-5)\cdots(n-11)(2n-11)!!}{2^{4}3}$. Therefore
$\tau_u\equiv0\pmod{4}$ if $v(n)=1$ and
$\tau_u\equiv\frac{n}{4}-2\pmod{8}$ if $v(n)\geq2$; if $u_3=3$, then
$v(\tau_u)=v((2n-8)!)-(n-9+6)-v((n-9)!)-v(3!)=v((n-4)(n-6)(n-8)-2\geq3$
unless $n\equiv2\pmod{8}$ where $v(\tau_u)=2$; if $u_3=2$, then we
can derive that
$\tau_u=-\frac{(2n-6)!}{2^{n-2}(n-6)!\cdot2}=-\frac{(n-3)(n-4)(n-5)(2n-7)!!}{4}$.
So $\tau_u\equiv\frac{3(n-4)}{4}\pmod{4}$ if $v(n)=1$ and
$\tau_u\equiv\frac{n-4}{4}\pmod{8}$ if $v(n)\geq2$; if $u_3=1$, then
$\tau_u=-\frac{(2n-4)!}{2^{n-1}(n-3)!}=-\frac{(n-2)(2n-5)!!}{2}$.
Thus $\tau_u\equiv-\frac{n-2}{2}\pmod{8}$ if $v(n)=1$ and
$\tau_u\equiv-3\frac{n-2}{2}\pmod{8}$ if $v(n)\geq2$. If $u_3=0$,
then $\tau_u=-\frac{(2n-3)!}{2n}$. By Lemma \ref{AHR Lem 4.1}, we
have $\tau_u\equiv-\frac{1}{2n}\pmod{4}$ if $v(n)=1$. By Lemma
\ref{001} (iii), we deduce that if $v(n)=2$, then
$\tau_u\equiv\frac{1+(-1)^{\frac{n+4}{8}}16}{2n}\equiv\frac{1}{2n}-2\pmod{8}$,
and  if $v(n)>2$, then
$$
\tau_u\equiv\frac{1+(-1)^{\frac{\frac{n}{2^{v(n)}}-1}{2}}
2^{v(n)+2}}{2n}\equiv\frac{1}{2n}+2\pmod{8}
$$
since $(-1)^{\frac{a-1}{2}}\equiv a \pmod 4$ for any odd number $a$.
Thus Theorem \ref{thm 3} is proved.
\end{proof}

With similar methods, but in a far more complicated fashion, we have
also obtained the following theorem. For reasons of brevity we delete the proof.

\begin{thm}\label{thm 4}
Let $n=m+l$, $l=k2^N$ with $2\nmid k$. Let $N\ge 3$ and $m\ge
2N+1$.
\begin{enumerate}
\item[(i)]
If $2\nmid m$, then
\begin{align*}
\frac{\widehat{B}_n}{n}\equiv&c_1^l\frac{\widehat{B}_m}{m}
+l(c_1^{n-12}c_3^4+c_1^{n-15}c^5_3+c_1^{n-5}c_2c_3+
c_1^{n-8}c_2c_3^2+c_1^{n-7}c_7)\\
&(-1)^{\frac{m+1}{2}}\frac{l}{2}(-c_1^n+c_1^{n-3}c_3+
c_1^{n-6}c_3^2+c_1^{n-9}c_3^3)
\pmod{2^{N+1}\mathbf{Z}_2[c_1,\ldots ,c_n]}.
\end{align*}
\item[(ii)]
If $2|m$ and $4\nmid m$, then
\begin{align*}
\frac{\widehat{B}_n}{n}\equiv&c_1^l\frac{\widehat{B}_m}{m}
+(l+\frac{l}{2mn})c_1^{n}+l(c_1^{n-9}c_3^3+c_1^{n-18}c_3^6
+c_1^{n-5}c_2c_3+c_1^{n-8}c_2c_3^2)\\
&-\frac{l}{2}c_1^{n-3}c_3 +\frac{3}{4}lc_1^{n-6}c_3^2+\theta
c_1^{n-12}c_3^4 \pmod{2^{N+1} \mathbf{Z}_2[c_1,\ldots ,c_n]},
\end{align*} where $\theta=-\frac{l}{2}$ for $N=3$ and $\theta=\frac{l}{2}$ for
$N\ge4$.
\item[(iii)]
If $4|m$ and $8\nmid m$, then
\begin{align*}
\frac{\widehat{B}_n}{n}\equiv&c_1^l\frac{\widehat{B}_m}{m} +(
l-\frac{l}{2mn})c_1^{n}+\frac{l}{2}c_1^{n-3}c_3
+\frac{l}{4}(c_1^{n-6}c_3^2+c_1^{n-12}c_3^4)\\
&+l(c_1^{n-8}c_2c_3^2+c_1^{n-5}c_2c_3)
\pmod{2^{N+1}\mathbf{Z}_2[c_1,\ldots ,c_n]}.
\end{align*}
\item[(iv)]
If $8|m$, then
\begin{align*}
\frac{\widehat{B}_n}{n}\equiv&c_1^l\frac{\widehat{B}_m}{m}
-\bigg(\frac{(2n-3)!!}{2n}-\frac{(2m-3)!!}{2m}\bigg)c_1^{n}+\frac{l}{2}c_1^{n-3}c_3
+\frac{5l}{4}c_1^{n-12}c_3^4\\
&+\frac{l}{4}c_1^{n-6}c_3^2+l(c_1^{n-24}c_3^8+c_1^{n-5}c_2c_3+c_1^{n-8}c_2c_3^2)
\pmod{2^{N+1}\mathbf{Z}_2[c_1,\ldots ,c_n]}.
\end{align*}
\end{enumerate}
\end{thm}

{\bf Acknowledgements.}\\

The authors would like to thank Professor Shparlinski and the anonymous referee for their
careful reading of the manuscript and for helpful comments and suggestions that improved
its presentation.

\end{document}